\documentclass[final]{siamltex}

\usepackage{times}
\usepackage{booktabs}
\usepackage{amsfonts}
\usepackage[colorlinks=true,linkcolor=red,filecolor=green,citecolor=red]{hyperref}
\usepackage[linesnumbered, ruled]{algorithm2e}
\usepackage{mathtools}

\usepackage[caption=false]{subfig}

\newtheorem{traceless}{Theorem}
\newtheorem{commute}[traceless]{Theorem}
\newtheorem{commute_id}{Corollary}
\newtheorem{symmetry}[traceless]{Theorem}
\newtheorem{symmetry_id}[commute_id]{Corollary}
\newtheorem{polynomials}[traceless]{Theorem}
\newtheorem{alphas}[traceless]{Theorem}
\newtheorem{hermitian}[commute_id]{Corollary}
\newtheorem{orth_algo}[commute_id]{Corollary}
\newtheorem{single_obc}[traceless]{Theorem}
\newtheorem{single_pbc_uneven}[traceless]{Theorem}
\newtheorem{double_obc_different}[traceless]{Theorem}
\newtheorem{double_obc_equal}[traceless]{Theorem}
\newtheorem{double_obc_gen}[commute_id]{Corollary}
\newtheorem{double_obc_div}[traceless]{Theorem}
\newtheorem{double_pbc_odd}[traceless]{Theorem}
\newtheorem{double_pbc_gen}[commute_id]{Corollary}
\newtheorem{triple_obc_odd}[traceless]{Theorem}
\newtheorem{triple_obc_two}[traceless]{Theorem}
\newtheorem{triple_obc_gen}[commute_id]{Corollary}

\newcommand*{\trace}{\mathsf{Tr}}
\DeclareMathOperator{\Span}{span}

\title{Towards a better Understanding of the Matrix Product Function Approximation Algorithm in Application to Quantum Physics}

\author{Moritz August\thanks{Department of Informatics, Technical University of Munich, 85748 Garching, Germany (august@in.tum.de)} \and Thomas Huckle\thanks{Department of Informatics, Technical University of Munich, 85748 Garching, Germany (huckle@in.tum.de)}}

\begin{document}

\maketitle

\begin{abstract}
	We recently introduced a method to approximate functions of Hermitian Matrix Product Operators or Tensor Trains that are of the form $\trace f(A)$.
	Functions of this type occur in several applications, most notably in quantum physics.
	In this work we aim at extending the theoretical understanding of our method by showing several properties of our algorithm that can be used 
	to detect and correct errors in its results. 
	Most importantly, we show that there exists a more computationally efficient version of our algorithm for certain inputs.
	To illustrate the usefulness of our finding, we prove that several classes of spin Hamiltonians in quantum physics fall into this input category.
	We finally support our findings with numerical results obtained for an example from quantum physics.
\end{abstract}



\maketitle

\section{Introduction}
Approximating functions of the form $\trace f(A)$ where $f:\mathbb{C}^{N \times N} \rightarrow \mathbb{C}^{N \times N}$ is analytic and smooth for large Hermitian matrices $A \in \mathbb{C}^{N \times N}$ is a problem of interest in areas such as computational chemistry, 
graph theory or quantum physics. 
In quantum physics, fundamental properties of states of many particle systems such as the entanglement entropy, the trace norm, heat capacity or expectation values are defined as functions of this form~\cite{sakurai1995modern}.

While computing $\trace f(A)$ is not challenging for small to medium size matrices, it becomes significantly harder for larger matrices of size $2^L$ with $L \gg 20$ where numerical diagonalization becomes computationally infeasible.
We have recently addressed this issue by presenting the first algorithm~\cite{august2017approximation} that is able to approximate such functions even for matrices of very high dimensionality via a combination of the global Krylov method with its connection to Gauss-type quadrature and the matrix product state (or tensor train) tensor decomposition scheme.
Our method constructs a basis $[U_1, \ldots, U_K]$ of $\Span \{A^0, A^1, \ldots, A^{K-1}\}$ where $U_i \in \mathbb{C}^{N \times N}$ and yields the projection $T_K$ of $A$ onto that space, which is used to approximate the desired function.

We have shown that our algorithm converges to the exact result or an arbitrarily good approximation thereof in the case of exact arithmetics and exact representation of the $U_i$.
While this result is instructive to understand the theoretical capability of the method, in practice the tensor decomposition is used to approximate $A$ and the $U_i$ and hence introduces an approximation error into the calculations.
Unfortunately, it is very difficult to analyze the propagation of such approximation errors over the course of a complete run of the algorithm and their influence on the final function approximation.
It is therefore important to gain a deeper understanding of theoretical properties of partial results of the computation, namely the $U_i$ and $T_i$, in order to be able to detect and possibly correct unwanted artifacts caused by the approximation errors.
In addition to that, we would of course like to avoid unnecessary operations that might introduce approximation errors and waste runtime whenever possible.
Thus, in this work we present several results regarding analytical properties of the $U_i$ and $T_i$ in the exact case and also show how these results can be used to obtain a more efficient version of our algorithm for a certain case of input matrices $A$.

While our algorithm is of general nature, as was hinted at above an important field of application can be found in numerical quantum physics.
Here, tensor networks have already been applied with great success for some time~\cite{garcia2006time, schollwock2011density, verstraete2008matrix, huckle2013computations, verstraete2006matrix, vidal2004efficient, vidal2003efficient, verstraete2004matrix, pirvu2010matrix} but so far a method to approximate functions of the type considered here was lacking.
Because of this, we will additionally present an analysis of possible use cases of our newly discovered algorithmic improvement for the application of spin Hamiltonians, an important problem in numerical quantum physics.

The rest of this work is structured as follows: in Section~\ref{algorithm}, we briefly introduce our method. 
Equipped with this knowledge, we present our analytical findings in Section~\ref{proofs}. 
In Section~\ref{hamiltonians}, we then present our analysis of possible applications of the previously introduced results in quantum physics.
Following this, we proceed to provide numerical evidence of the correctness of our claims regarding the existence of an improved version of our algorithm in Section~\ref{numerics}.
Finally, we conclude this work in Section~\ref{conclusion}.

\section{The Algorithm}
\label{algorithm}
\begin{algorithm}[t!]
\setlength{\leftskip}{10pt}
\setlength{\skiprule}{10pt}
\caption{Approximation Algorithm}
\label{approx}
    \SetKwInOut{Input}{Input}
    \SetKwInOut{Output}{Output}
    \SetKwFunction{multOpt}{multiply}
    \SetKwFunction{addOpt}{sum}
    \SetKwFunction{contract}{innerProduct}
	\SetKwFunction{scalMult}{multiply}
	\SetKwFunction{spec}{spectralDecomposition}
	\SetKwFunction{checks}{checkStop}

    \Input{MPO $A[D_{A}] \in \mathbb{C}^{N \times N}$, Starting orthogonal MPO $U[D_{init}] \in \mathbb{C}^{N \times N}$, Number of Dimensions $K$, Maximal Bond-Dimension $D_{max}$, Stopping Criteria $\mathcal{S}$}    
	%$\beta_1 \leftarrow \sqrt{\contract (U_1, U_1)}$ \;    
    %$U_1 \leftarrow \scalMult(1/ \beta_1, U_1$) \;
    $U_0 \leftarrow 0$ \;
    $V_0 \leftarrow U$ \;
    $D \leftarrow D_{init}$ \;
    \For{$i \leftarrow 1 ; i \leq K $}{
	$\beta_i \leftarrow \sqrt{\contract (V_{i-1}, V_{i-1})}$ \;
    \If{$\beta_i = 0$}{break \;}
    $U_i \leftarrow  \scalMult(1 / \beta_i,V_{i-1})$ \;    
    $D \leftarrow \min(D_{max}, D \cdot D_{A})$ \;
    $V_i \leftarrow \multOpt(A, U_{i}, D)$ \;
	$D \leftarrow \min(D_{max}, D + D_{U_{i-1}})$ \;    
    $V_i \leftarrow \addOpt(V_i, - \beta_i U_{i-1}, D)$ \;
    $\alpha_i \leftarrow \contract(U_i, V_i)$ \;
	$D \leftarrow \min(D_{max}, D + D_{U_{i}})$ \;    
    $V_i \leftarrow \addOpt(V_i, -\alpha_i U_{i}, D)$ \;
    
    $V_T \Lambda_T V_T^* \leftarrow \spec(T_i)$ \;
    $\mathcal{G}f \leftarrow \beta_1^2 e_1^T V_T f(\Lambda_T) V_T^* e_1$ \;
    \If{$\checks(\mathcal{G}f, \Lambda_T, \mathcal{S})$}{break \;}
    }
    
  	\Output{Approximation $\mathcal{G}f$ of $\trace f(A)$}
\end{algorithm}

	As we have stated above, our goal is to approximate functions of the form $\trace f(A)$. 
	For smaller to medium sized matrices, there already exists a well-established method to achieve this in performing a Gauss-type quadrature via the projection of $A$ onto a Krylov space starting with a carefully chosen initial vector~\cite{bellalij2015bounding, bai1996some, golub1994matrices, tang2012probing, reichel2015generalized}. 

	For symmetric or Hermitian matrices, the global Lanczos method recently was introduced as a formulation of the classical Lanczos method in terms of basis \emph{matrices} with the Frobenius inner product defined as
\begin{equation*}
\langle U_i, U_j \rangle = \trace U_i^* U_j
\end{equation*}
and $U_i, U_j \in \mathbb{C}^{N \times M}$~\cite{bellalij2015bounding}. 
Note that this inner product acts on entire matrices, meaning that the global Lanczos method differs from other block Krylov methods in that it only orthogonalizes whole matrices in contrast to the individual columns therein.
The algorithm iteratively builds up a basis $\mathbf{U}_i = [U_1, U_2, \ldots, U_i]$ of the Krylov space and yields the partial global Lanczos decomposition
\begin{equation*}
A \mathbf{U}_i = \mathbf{U}_i \tilde{T}_i + \beta_{i+1}U_{i+1}E_i^T
\end{equation*}
where $\tilde{T}_i = T_i \otimes I_M \in \mathbb{R}^{iM \times iM}$ and $E_i^T = [\mathbf{0},\cdots,\mathbf{0},I_M] \in \mathbb{R}^{M \times iM}$. 
It was shown that the connection between the Lanczos algorithm and Gauss quadrature extends to the global Lanczos method. 
Hence, for an initial matrix $U \in \mathbb{C}^{N \times M}$ it in general holds
\begin{align*}
	\trace U^{*} f(A) U = \int f(\lambda) d\mu(\lambda) \approx \trace f(A)
\end{align*}
with $\mu$ being the distribution in the Riemann-Stieltjes integral generated by $U$ and $ $$\trace U^* f(A) U$ yields a Gauss quadrature of $\trace f(A)$.
However, for the algorithm to remain computationally efficient or at least more efficient than computing the eigenvalue decomposition of $A$ directly, it is required that $M \ll N$ and thus $U$ can not be orthogonal/unitary. 
This implies that $U^*U \neq I$ and consequentially the method does in general not converge to the exact result so that sampling over multiple starting matrices is required if the approximation error is to be minimized. 
Additionally, for very large matrices even the computation of the aforementioned inner product becomes infeasible.

To allow for approximations of larger matrices, we reformulated the global Lanczos algorithm in terms of matrix product operators (MPO) which support all basic linear algebra operations. 
A matrix product operator decomposes a matrix $A \in \mathbb{C}^{N \times N}$ such that 
\begin{equation*}
A_{ij} = A_{i_1\dots i_Lj_1\dots j_L} = \trace C_1^{i_1j_1}C_2^{i_2j_2} \dotsm C_L^{i_Lj_L}
\end{equation*}
where the indices $i,j$ are split up into $i_1,\ldots,i_L$ and $j_1,\ldots,j_L$ respectively and are called the physical indices. 
We here assume all physical indices to be of equal dimension $d=2$ which corresponds to the assumption that $N = d^L$ but our results also carry over to the case of varying $d_k$. 
The $C_k \in \mathbb{C}^{D_k \times D_k \times d \times d}$ are called the core tensors where $D_k$ is referred to as the bond dimension or auxiliary index. 
We additionally define the bond dimension of the MPO $D=\max_k D_k$ to be the maximal bond dimension over all core tensors.
It follows that $C_k^{i_k, j_k}$ is a matrix of size at most $D \times D$ and the right-hand side of the above equation yields a scalar.
The matrix product operator representation of a matrix requires $Ld^2D^2$ parameters which, depending on the choice of $D$, either poses an approximation or suffices for an exact representation.
While naturally the accuracy of the approximation increases with growing $D$, it is commonly chosen such that $Ld^2D^2 \in \mathcal{O}(poly(L))$ and thus yields an efficient,i.\ e.\ polynomial in contrast to exponential in $L$, representation. 
It has been found that such choices of $D$ often suffice for a good approximation.
Note that especially in numerical quantum physics, it is possible and common to formulate matrices of interest, such as Hamiltonians, directly as matrix product operators and perform computations on them so that an explicitly stored matrix is at no point required.
While explaining the decomposition in more detail exceeds the scope of this section, we refer the interested reader to the overview articles~\cite{schollwock2011density, perez2006matrix, grasedyck2013literature}.

Our algorithm can thus be perceived as the global Lanczos algorithm reformulated for matrix product operators. 
However, the differences between the methods extend beyond the different possible sizes of the input matrices. 
In our method, we choose the identity matrix written as a matrix product operator as initial matrix $U$. 
This can be done exactly with a minimal bond dimension of $D=1$. 
Thus, we are theoretically and practically able to start our computation with an orthogonal/unitary matrix of the same dimensionality as $A$ which is not feasible in the original method as discussed above. 
Hence, instead of the previous equation 
\begin{equation*}
	\trace U^* f(A) U \approx \trace f(A)
\end{equation*}
with $U \in \mathbb{C}^{N \times M}$ and $M \ll N$, in our method it holds

\begin{equation*}
	\trace U^* f(A) U = \trace f(A)
\end{equation*}
and $U \in \mathbb{C}^{N \times N}$ is orthogonal/unitary.
Our algorithm consequentially computes a Krylov basis of $\Span\{ A^0,A^1,A^2,\cdots,A^{K-1} \}$ and the approximation error in $\trace f(A)$ is controlled by the maximal Krylov dimension $K$ and the maximally allowed bond dimension $D_{max}$.
This implies that our algorithm produces an approximation of the exact result which can in principle be made arbitrarily accurate by increasing $K$ and most importantly $D_{max}$. 
The method is shown in Algorithm~\ref{approx}. 
Note that the subfunctions \texttt{multiply} and \texttt{sum} involve solving an optimization problem to find a good representation of the result for a given bond dimension $D$.
The respective optimization algorithms employ the sweeping scheme typical for tensor network optimizations where the individual core tensors are optimized sequentially in a dynamic programming fashion.
In conclusion, our algorithm poses the first method to approximate functions of the form $\trace f(A)$ of matrices of size significantly larger then $2^{20}$  and has no analytical lower bound on the approximation error.

\section{Analytical Results}
\label{proofs}
As in practice we must impose $D \in \mathcal{O}(poly(L))$ to remain computationally feasible, it is important to be able to detect when the approximations made lead to unreasonbly large errors in the computed basis matrices $U_i$ and the projections $T_i$ of $A$. 
This is especially important as it is clear that since the basis matrices are computed iteratively and depending on the previously computed ones, any error introduced in a given iteration will be propagted and influence all following iterations.
Additionally, since the basic arithmetic operations are comparibly costly in the matrix product operator domain and can infuse errors into the computation, we would like to reduce their number whenever possible.
In this section we will thus present some analytical results on properties of the basis matrices $U_i$ and the projection $T_i$ that can be checked for during a run of Algorithm~\ref{approx} and that finally give rise to a more efficient version of our method for a special class of inputs.
All following proofs and corollaries assume the matrix $A \in \mathbb{C}^{N \times N}$ to be Hermitian and the initial matrix $U$ to be of the same dimensions as $A$. 
Note that we additionally assume exact arithmetics and exact representation of the $U_i$ as our aim is to derive insight about the algorithm's ideal behaviour to be able to detect deviations from it.
As we will make use of these equations in all following proofs of this section, it is worthwhile to explicitly state the update rules
\[
U_{n+1} = \left(AU_n - \alpha_{n+1} U_n - \beta_n U_{n-1}\right)/\beta_{n+1}
\]
and
\[
\alpha_{n+1} = \trace U_n^* U_{n+1}
\]
as implied by Algoritm~\ref{approx}.
Note also that our analysis is based on the Frobenius inner product that was, to the best of the author's knowledge, for the first time employed in the original global Lanczos algorithm~\cite{elbouyahyaoui2009algebraic, bellalij2015bounding} and our consecutive work~\cite{august2017approximation}.
Combined with the fact that we assume $A$ and $U$ to be of equal dimension this shows that our results are complementary to other work on structured matrices in Krylov type algorithms~\cite{mastronardi2005lanczos, fassbender2006structured, browne2009lanczos, benner2017fast, shao2016structure, mehrmann2001structure, voss2000symmetry}.
Given these assumptions it is also clear that existing analyses of standard Lanczos type algorithms working on column vectors can not cover the following results.

We begin by stating a result about the inheritance of tracelessness of the basis matrices from the input $A$.
Since it is possible to efficiently compute the trace of a given MPO, this property of the $U_i$ can be efficiently checked for and, if desired, enforced during a run of the algorithm. 

\begin{traceless}
If $A \in \mathbb{C}^{N \times N}$ is traceless and $U_0 \in \mathbb{C}^{N \times N} = I_N/\beta_0$, all basis matrices $U_i \in \mathbb{C}^{N \times N}, i \in \{1,\cdots,K\}$ as constructed by the algorithm are traceless.
\end{traceless}
\begin{proof}
We prove the statement by induction over the iteration number $n$ of the algorithm.
For $n=1$, it is easy to see that $\trace U_1 = \trace A/(\beta_1 \beta_0) = 0$ as $\alpha_1  = \trace A/\beta_0^2 = 0$.
We now obtain for $n=2$ that $\trace U_2 = (\beta_1\beta_0 \trace U_1^*U_1 - \alpha_2/(\beta_1\beta_0) \trace A - \beta_1/\beta_0 \trace I)/\beta_2 = 0$.
This establishes the inductive basis.

In the inductive step for $n \geq 2$ we then have

\begin{align*}
\trace U_{n+1} =& (\trace AU_n - \alpha_{n+1} \trace U_n - \beta_n \trace U_{n-1})/\beta_{n+1} \\
=& (\trace AU_n - \alpha_{n+1} 0 - \beta_n 0)/\beta_{n+1} \\
=& 0
\end{align*}

\end{proof}

Next, we present a result about the commutation relation of $A$ and the $U_i$ that will also become useful for proving subsequent statements.
As for the previous result, this property can be efficiently checked for during an execution of the algorithm assuming MPO representation of the $U_i$ by computing the Frobenius norm, which is efficiently computable for MPOs, of the distance between $AU_i$ and $U_iA$.

\begin{commute}
If $A \in \mathbb{C}^{N \times N}$ commutes with $U_0 \in \mathbb{C}^{N \times N}$, A commutes with all basis matrices $U_i \in \mathbb{C}^{N \times N}, i \in \{1,\cdots,K\}$ as constructed by the algorithm.
\end{commute}
\begin{proof}
We again prove the statement by induction over the iteration number $n$.
To start, we note that $[U_1, A] = \left( (A A U_0 - \alpha_1 A U_0) - (A A U_0 - \alpha_1 A U_0) \right) /\beta_1 = 0$.

In the inductive step for $n \geq 1$, it is now straight forward to see that
\begin{align*}
[U_{n+1}, A] =& \left[ (AU_nA - \alpha_{n+1} U_nA - \beta_n U_{n-1}A) - (AAU_n - \alpha_{n+1} A U_n - \beta_n A U_{n-1}) \right]/\beta_{n+1} \\
=& \left[ (AAU_n - \alpha_{n+1} AU_n - \beta_n A U_{n-1}) - (AAU_n - \alpha_{n+1} A U_n - \beta_n A U_{n-1}) \right]/\beta_{n+1} \\
=& 0
\end{align*}

\end{proof}
\begin{commute_id}
	We note that any $A \in \mathbb{C}^{N \times N}$ commutes with $I_N$. 
	Thus it follows that the above statement holds for Algorithm~\ref{approx}.
\end{commute_id}

The following finding addresses the symmetry properties of the basis matrices in relation to the input $A$ and the initial basis matrix $U_0$.
These symmetry properties can as well be tested efficiently in a manner similar to the way the commutation relation can be checked since the permutation matrix $J$, like $I$, permits a formulation in MPO format with minimal bond dimension.
Additionally, symmetries could be leveraged to obtain more efficient representations of the $U_i$ by reflecting them in the structure of the MPOs and thus obtaining more efficient and stable expressions.

\begin{symmetry}
If $A \in \mathbb{R}^{N \times N}$ is symmetric, persymmetric or centrosymmetric and $U_0 \in \mathbb{R}^{N \times N}$ is symmetric, persymmetric or centrosymmetric and commutes with $A$, all basis matrices $U_i \in \mathbb{R}^{N \times N}, i \in \{1,\cdots,K\}$ as constructed by the algorithm are symmetric, persymmetric or centrosymmetric.
\end{symmetry}
\begin{proof}
As for the above statements, we prove this statement by induction over the iteration number $n$.
To establish the inducte basis, we observe that for the case of symmetry 
$U_1^T = ((AU_0)^T - \alpha_1 U_0^T)/\beta_1 = (A U_0 - \alpha_1 U_0)/\beta_1 = U_1$. 
Likewise, we find that $U_1J = (AU_0 - \alpha_1 U_0)/\beta_1 J = J ((A U_0)^T - \alpha_1 U_0^T)/\beta_1 = J U_1^T$ for persymmetry and finally 
$JU_1 = J (AU_0 - \alpha_1 U_0)/\beta_1 = (A U_0 - \alpha_1 U_0)/\beta_1 J = U_1J$ for centrosymmetry.

For $n \geq 1$, we can now make the inductive step by 
\begin{align*} 
U_{n+1}^T =& ((AU_n)^T - \alpha_{n+1} U_n^T - \beta_n U_{n-1}^T)/\beta_{n+1}\\
=& (U_n A - \alpha_{n+1} U_n - \beta_n U_{n-1})/\beta_{n+1} \\
=& U_{n+1}
\end{align*}
for symmetry,
\begin{align*}
U_{n+1} J =& (A U_n - \alpha_{n+1} U_n - \beta_n U_{n-1})/\beta_{n+1} J\\
=& (J A^T U_n^T - \alpha_{n+1} J U_n^T - \beta_n J U_{n-1}^T)/\beta_{n+1} \\
=& J U_{n+1}^T
\end{align*} 
for persymmetry and 
\begin{align*}
JU_{n+1} =& J (AU_n - \alpha_{n+1} U_n - \beta_n U_{n-1})/\beta_{n+1} \\
=& (A U_n J - \alpha_{n+1} U_n J - \beta_n U_{n-1} J)/\beta_{n+1} \\
=& U_{n+1} J
\end{align*}
for centrosymmetry.
\end{proof}

\begin{hermitian}
As can be easily verified based on its proof, the above statement extends to the case of hermiticity, perhermiticity and centrohermiticity when $A, U_i \in \mathbb{C}^{N \times N}, i \in \{0,\cdots,K\}$.
\end{hermitian}

\begin{symmetry_id}
We note that the matrix $I_N$ is symmetric, persymmetric and centrosymmetric as well as hermitian, perhermitian and centrohermitian. 
Hence the above statements hold for Algorithm~\ref{approx}.
\end{symmetry_id}

We now turn our attention to a description of the $U_i$ in terms of polynomials as might seem natural given the underlying Lanczos algorithm. 
However, we restrict our analysis to the particular case where all $\alpha_i = 0$
to obtain a result that will become important in the proof of the subsequent statement.

\begin{polynomials}
If for $A \in \mathbb{C}^{N \times N}$ all $\alpha_i=0, i \in \{1,\cdots,K\}$ as computed by the algorithm and $U_0 \in \mathbb{C}^{N \times N} = I_N/\beta_0$, then all $U_i \in \mathbb{C}^{N \times N}, i \in \{1,\dots,K\}$ are polynomials of the form $\sum_{j \in 2 \mathbb{N}_0 \leq i} c_j A^j$ if $i$ is even and $\sum_{j \in 2 \mathbb{N}-1 \leq i} c_j A^j$ if $i$ is odd.
\end{polynomials}
\begin{proof}
We again prove the statement by induction over the iteration number $n$.
We establish the inductive basis by observing that $U_0 = A^0/\beta_0$, $U_1 = A/(\beta_1 \beta_0)$, $U_2 = A^2/(\beta_2 \beta_1 \beta_0) - \beta_1/(\beta_2 \beta_0) A^0$ and $U_3 = A^3/(\beta_3 \beta_2 \beta_1 \beta_0) - (\beta_2^2 + \beta_1^2)/(\beta_3 \beta_2 \beta_1 \beta_0) A$ are all polynomials of the types specified above.

In the inductive step, we then find for even $n$ that 
\begin{align*}
U_{n+1} =& (AU_n - \beta_n U_{n-1})/\beta_{n+1} \\
=& \left(A\sum_{j \in 2 \mathbb{N}_0 \leq n} c_j A^j - \beta_n \sum_{j \in 2 \mathbb{N}-1 \leq n-1} d_j A^j\right)/\beta_{n+1} \\
=& \left(\sum_{j \in 2 \mathbb{N}-1 \leq n+1} c_j A^j - \beta_n \sum_{j \in 2 \mathbb{N}-1 \leq n-1} d_j A^j\right)/\beta_{n+1} \\
=& \sum_{j \in 2 \mathbb{N}-1 \leq n+1} (c_j-\beta_n d_j)/\beta_{n+1} A^j
\end{align*}

and analogously for odd $n$
\begin{align*}
U_{n+1} =& (AU_n - \beta_n U_{n-1})/\beta_{n+1} \\
=& \left(A\sum_{j \in 2 \mathbb{N} - 1 \leq n} c_j A^j - \beta_n \sum_{j \in 2 \mathbb{N}_0 \leq n-1} d_j A^j\right)/\beta_{n+1} \\
=& \left(\sum_{j \in 2 \mathbb{N}_0 \leq n+1} c_j A^j - \beta_n \sum_{j \in 2 \mathbb{N}_0 \leq n-1} d_j A^j\right)/\beta_{n+1} \\
=& \sum_{j \in 2 \mathbb{N}_0 \leq n+1} (c_j-\beta_n d_j)/\beta_{n+1} A^j
\end{align*}

where we defined $d_{n+1} \coloneqq 0$.
\end{proof}

We can now use this result to obtain a more profound insight.

\begin{alphas}
If $A \in \mathbb{C}^{N \times N}$ has a spectrum that is point-wise symmetric around zero and $U_0 \in \mathbb{C}^{N \times N} = I_N/\beta_0$, all the $\alpha_i, i \in \{1,\cdots,K\}$ as computed by the algorithm are zero.
\end{alphas}
\begin{proof}
As done previously, we prove this statement by induction over the iteration number $n$.
We start by observing that by assumption $\trace A = 0$ and hence $\alpha_1 = \trace A/\beta_0^2 = 0$.
Consequentially, we have that $\alpha_2 = \trace (AU_1)^*U_1 = 1/(\beta_1 \beta_0)^2 \trace A^3 = 0$ which is our inductive basis.

Now for the inductive step, we begin by noting that it follows from the inductive hypothesis that all $U_i, i \in \{1,\cdots,n\}$ are polynomials of the form defined in the previous statement. 
Then for odd $n$, it follows that

\begin{align*}
\alpha_{n+1} =& \trace U_n U_n^* A \\
=& \trace \left(\sum_{j \in 2 \mathbb{N}-1 \leq n} c_j A^j\right) \left(\sum_{j \in 2 \mathbb{N}-1 \leq n} c_j A^{*^j}\right)A \\
=& \sum_{j \in 2 \mathbb{N}+1 \leq 2n+1} c_j \trace A^j \\
=& 0.
\end{align*}

Analogously, it follows for even $n$ that

\begin{align*}
\alpha_{n+1} =& \trace U_n U_n^* A \\
=& \trace \left(\sum_{j \in 2 \mathbb{N}_0 \leq n} c_j A^j\right) \left(\sum_{j \in 2 \mathbb{N}_0 \leq n} c_j A^{*^j}\right)A \\
=& \sum_{j \in 2 \mathbb{N}-1 \leq 2n+1} c_j \trace A^j \\
=& 0.
\end{align*}
\end{proof}

From this finding, we finally obtain the following corollary.

\begin{orth_algo}
	For $A \in \mathbb{C}^{N \times N}$ having a spectrum point-wise symmetric around zero, Algorithm~\ref{approx} produces a bidiagonal matrix \begin{equation*}
T_K = \left[ \begin{matrix}
0 & \beta_1 & & \mathbf{0}\\
\beta_1 & 0 & \ddots & \\
& \ddots & \ddots & \beta_K \\
\mathbf{0} & & \beta_K & 0 \end{matrix} \right].
\end{equation*}
\end{orth_algo}
This insight yields on one hand a more efficient version of the algorithm as each $U_i$ is in theory guaranteed to be orthogonal to $U_{i-1}$ and hence only one orthogonalization has to be performed in each iteration of the algorithm. 
Since orthogonalizations of MPOs require solving an optimization problem and are hence significantly more computationally demanding than the orthogonalization of full matrices, avoiding them results in a measurable reduction of the runtime as we will illustrate later.
%Furthermore, since $T_K$ is symmetric, we also know that the inner product of $U_i$ and $U_{i-2}$ must be $\beta_{i-1}= ||U_{i-1}||_F$ and thus we do not need to compute the inner product between $U_i$ and $U_{i-2}$ before the orthogonalization.
On the other hand, we obtain yet another means of checking for the effect of truncation errors by monitoring the magnitude of the $\alpha_i$ when it is known they must be zero.
It is of course also possible to still orthogonalize against the previous two basis MPOs but always set $\alpha_i=0$ to increase the approximations accuracy.
%and computing the difference between $\beta{i-1}$ and the actual overlap between $U_I$ and $U_{i-2}$.
However, it is worth noting that the condition of the $\alpha_i$ being equal to zero is necessary but not sufficient for the overall approximation of $\trace f(A)$ to be accurate. The deviation from zero of the $\alpha_i$ does not allow us to draw strong conclusions about the accuracy of the approximation of $\trace f(A)$.
To illustrate this point, we add a few remarks.
\begin{itemize}
	\item Although it seems reasonable to assume that when the $\|\alpha_i \|$ remain small the approximated values of $\beta_i$ are also close to their true values, we have no way of inferring the error in the $\beta_i$ from the error in the $\alpha_i$. 
		This is mainly the case because we do not have access to the true values of the $\beta_i$ and we believe it not to be possible to establish an analytical practically relevant connection between both errors in our algorithm, especially when truncations come into play.
	\item If however we find some $\alpha_i$ to be significantly larger in magnitude than zero, we know that the respective two basis MPOs are not orthogonal as they should be which usually leads to the reocurrence of previously observed approximated eigenvalues.
Although we know that the accuracy of the overall approximation will suffer from this, it is unfortunately not possible to make a more precise statement as we cannot tell in detail how a deviation from zero relates to amount and magnitude of such `ghost' eigenvalues and again we do not know the error in the $\beta_i$.
	\item Although in principle one could counter growing magnitudes of the $\alpha_i$ by increasing $D$ of the basis MPOs and (re)orthogonalizing, in practice the bond dimension of the basis MPOs reaches $D_{max}$ after already a few iterations and by definition we cannot exceed this value.
	Still, it would be possible to either restart a failed run with a larger maximal bond dimension or increase it dynamically until the $\alpha_i$ become small enough.
	The latter approach however could be argued to defeat the purpose of the $D_{max}$ parameter.
\end{itemize}

\section{Spectra of Hamiltonians}
\label{hamiltonians}
The results obtained in the previous section naturally raise the question what kinds of matrices exhibit a spectrum point wise symmetric around zero and how many cases of relevance there are.
While we cannot give a general answer to this question, we can provide a partial answer for a specific application, namely spin systems in quantum physics. 
These systems are often studied analytically and numerically because they exhibit interesting physical phenomena while still allowing for the derivation of mathematically rigorous results and comparably efficient simulations by tensor network approaches.

Spin systems are described by their corresponding Hamiltonians which for open boundaries and interactions between direct neighbours take the form 
\[
H_{OBC} = \sum_{(i,\alpha) \in \mathcal{I}} \sum^{L-i}_{j=0} h_{ij\alpha} I^{\otimes j} \otimes \sigma_{\alpha}^{\otimes i} \otimes I^{\otimes L-i-j}
\]
where $L \in \mathbb{N}$ is the number of spin particles and $\mathcal{I}$ is a set of tuples $(i, \alpha) \in \mathbb{N}_L \times \{x,y,z\}$ denoting the number of consecutive applications of $\sigma_{\alpha}$.
In this case, $\sigma_{x,y,z}$ denote the Pauli matrices 
\begin{equation*}
\sigma_x = \left[ \begin{matrix}
0 & 1 \\
1 & 0 \\ \end{matrix} \right],
\sigma_y = \left[ \begin{matrix}
0 & -i \\
i & 0 \\ \end{matrix} \right],
\sigma_z = \left[ \begin{matrix}
1 & 0 \\
0 & -1 \\ \end{matrix} \right].
\end{equation*} 
and the $h_{ij\alpha} \in \mathbb{R}$ simply are scaling constants.
Similarly, the case of closed or periodic boundaries is expressed as
\[
H_{PBC} = \sum_{(i,\alpha) \in \mathcal{I}} \sum^{L-i}_{j=0} h_{ij\alpha} I^{\otimes j} \otimes \sigma_{\alpha}^{\otimes i} \otimes I^{\otimes L-i-j} + \sum_{(i,\alpha) \in \mathcal{I}} \sum^{i-1}_{k=1} h_{ik\alpha} \sigma_{\alpha}^{\otimes k} \otimes I^{\otimes L-i} \otimes \sigma_{\alpha}^{\otimes i-k}.
\]
We will in the following denote the individual terms in the sums of the Hamiltonians as interaction terms and refer to the products of multiple Pauli matrices inside these interaction terms as blocks. 
This terminology is derived from the fact that each term describes the interaction between the particles at whose position there is a Pauli operator in the product.
While the formulations introduced above naturally do not describe all possible Hamiltonians, they cover many interesting cases which are furthermore treatable via tensor network methods. 
This typically gets much more difficult for cases of arbitrary and long-range interaction patterns, which are not covered by the above expressions.

Now, one sufficient condition for the existence of a point-wise symmetric spectrum around zero looks as follows: for a given $H \in \mathbb{C}^{2^L \times 2^L}$, there exists a unitary and Hermitian matrix $R$ of equal size, such that
\[
RH = -HR
\]
and consequentially by the standard eigenvalue formulation $Hv=\lambda v$ it holds that 
\[
H(Rv)=-\lambda (Rv).
\]
In the following, we will make statements about the existence of such an $R$ for several classes of spin Hamiltonians.
Although absence of such an $R$ does not imply that the Hamiltonian in question does not exhibit a point symmetric spectrum around zero, the above formulation captures a large class of possible symmetries and thus poses a relevant albeit not final characterization of point symmetric spectra.
Note that in quantum physics it is already known that one can make use of the rotation transformation properties of spin operators to change the sign of particular terms in a Hamiltonian~\cite{sakurai1995modern}. 
However, here we tackle the problem of changing all terms, from now on also called interaction terms, in a Hamiltonian to relate different eigenvalues/-states to each other and our focus lies on formally defining classes of spin Hamiltonians for which such an $R$ exists and which hence are valid inputs for our improved algorithm.
A related mathematical discussion for general square matrices was presented by Fassbender et.\ al.~\cite{fassbender2006structured}.

Before we start, we remind ourselves that the Pauli matrices are Hermitian and unitary and that each pair of Pauli matrices anticommutes such that for $\alpha, \beta \in \{x,y,z\}$ it holds
\[
\{\sigma_{\alpha},\sigma_{\beta}\} = 2\delta_{\alpha, \beta}I.
\]
Furthermore, we note that the Kronecker product of Hermitian and unitary matrices is again Hermitian and unitary. 
These properties will be used in all following proofs.

We start by considering Hamiltonians with open boundaries and neighbour interactions of arbitrary length for a single Pauli operator.
\begin{single_obc}
	\label{sing_obc}
	For every spin Hamiltonian with open boundaries of the form
	\[
	H_{OBC, \alpha, i} = \sum_{j=0}^{L-i} h_j I^{\otimes j} \otimes \sigma_{\alpha}^{\otimes i} \otimes I^{\otimes L -i -j}
	\]
	where $\alpha \in \{x,y,z\}$ and $i, L \in \mathbb{N}$ and $i \leq L$, there exists a unitary $R \in \mathbb{C}^{2^L \times 2^L}$ such that $RH_{OBC, \alpha, i} = -H_{OBC, \alpha, i}R$.
\end{single_obc}
\begin{proof}
	%We note that the Pauli matrices are unitary and that they anticommute pairwise. 
	%Thus for every $\alpha \in \{x,y,z\}$ we can choose an $\alpha^{\prime} \in \{x,y,z\} \setminus \{\alpha\}$ such that $\sigma_{\alpha} \sigma_{\alpha^{\prime}} = -\sigma_{\alpha^{\prime}} \sigma_{\alpha}$.
	We can construct $R=\left(I^{\otimes i-1} \otimes \sigma_{\alpha^{\prime}}\right)^{\otimes L/i} \otimes I^{\otimes L\%i}$ with $\alpha^{\prime }\neq \alpha$  as we need only apply one $\sigma_{\alpha^{\prime}}$ for each block $\sigma_{\alpha}^{\otimes i}$ to change the sign in every term of the sum in $H_{\alpha, i}$ and thereby ultimately the sign of $H_{\alpha, i}$ itself. Hereby, $\otimes^{L/i}$ denotes the repition of the given expression for $L/i$ times whereas $I^{\otimes L \% i}$ simply refers to a `padding' of $R$ to the required length $L$.  
	%By observing that the Kronecker product of unitary matrices and $I$ is again unitary, we finally obtain that $R$ is unitary.
\end{proof}

While for the case of open boundaries and a single Pauli matrix the statement is quite universal, the additional structure introduced by periodic boundaries forces us to restrict the statement to odd interactions lengths.

\begin{single_pbc_uneven}
	For every spin Hamiltonian with periodic boundaries of the form
	\[
	H_{PBC, \alpha, i} = \sum_{j=0}^{L-i} h_j I^{\otimes j} \otimes \sigma_{\alpha}^{\otimes i} \otimes I^{\otimes L -i -j} + \sum_{k=1}^{i-1} h_k \sigma_{\alpha}^{\otimes k} \otimes I^{\otimes L-i} \otimes \sigma_{\alpha}^{i-k}
	\]
	where $\alpha \in \{x,y,z\}$ and $i \in 2\mathbb{N}-1$, $L \in \mathbb{N}$ and $i \leq L$, there exists a unitary $R \in \mathbb{C}^{2^L \times 2^L}$ such that $RH_{PBC, \alpha, i} = -H_{PBC, \alpha, i}R$.
\end{single_pbc_uneven}

\begin{proof}	
	%We note that the Pauli matrices anticommute pairwise and that they are unitary. Hence, for every $\alpha \in \{x,y,z\}$ we can choose an $\alpha^{\prime} \in \{x,y,z\} \setminus \{\alpha\}$ such that $\sigma_{\alpha} \sigma_{\alpha^{\prime}} = -\sigma_{\alpha^{\prime}} \sigma_{\alpha}$.
	By defining $R = \sigma_{\alpha^{\prime}}^{\otimes L}$ with $\alpha^{\prime }\neq \alpha$ we obtain an odd number of sign changes in every term of the sum in $H_{PBC, \alpha, i}$ inducing a sign change of $H_{PBC, \alpha, i}$. 
	%As $R$ is the Kronecker power of a unitary matrix, $R$ furthermore is unitary.
\end{proof}

These two statements together show that a large subset of spin Hamiltonians with interaction terms involving only one particular Pauli matrix exhibit a point symmetric spectrum around zero. 
Not surprisingly, the situation becomes more involved when considering Hamiltonians with up to two different Pauli operators and differing interaction lengths.
We first examine the case of interaction terms involving two differing Pauli operators.

\begin{double_obc_different}
	For every spin Hamiltonian with open boundaries of the form
	\[
	H_{OBC, \alpha,\beta, i,k} = \sum_{j=0}^{L-i} h_{\alpha j} I^{\otimes j} \otimes \sigma_{\alpha}^{\otimes i} \otimes I^{\otimes L -i -j} + \sum_{l=0}^{L-k} h_{\beta l} I^{\otimes l} \otimes \sigma_{\beta}^{\otimes k} \otimes I^{\otimes L -k -l}
	\]
	where $\alpha, \beta \in \{x,y,z\}$, $\alpha \neq \beta$, $i, k, L \in \mathbb{N}$ and $i, k \leq L$, there exists a unitary $R \in \mathbb{C}^{2^L \times 2^L}$ such that $RH_{OBC, \alpha, \beta, i, k} = -H_{OBC, \alpha, \beta, i, k}R$.
	
\end{double_obc_different}

\begin{proof}
	%As in the previous proofs, we begin by noting that the Pauli matrices anticommute pairwise and that they are unitary. 
We have to distinguish two cases regarding the relation of $i$ and $k$.

\textbf{Case $i = k$}

In this case there is a $\gamma \in \{x,y,z\} \setminus \{\alpha, \beta\}$ such that $\sigma_{\alpha}\sigma_{\gamma} = -\sigma_{\gamma}\sigma_{\alpha}$ and $\sigma_{\beta}\sigma_{\gamma} = -\sigma_{\gamma}\sigma_{\beta}$. 
Hence, we can again define $R = \left( I^{\otimes i-1} \otimes \sigma_{\gamma} \right)^{\otimes L / i} \otimes I^{\otimes L\%i}$ to obtain a unitary that induces a sign change in every block $\sigma_{\alpha}^{\otimes i}$ and $\sigma_{\beta}^{\otimes k}$ and hence changes the sign of $H_{OBC, \alpha,\beta, i,k}$.

%\textbf{Case $i \neq k, i>1, k>1$}
%
%Let w.l.o.g. $i > k$. 
%Then, we can define $R = \left( I^{k-1} \otimes \sigma_{\alpha} \otimes I^{i-k-1} \otimes \sigma_{\beta} \right)^{\otimes L \% i}$. 
%Due to the unitarity of the Pauli matrices, $\sigma_{\alpha}$ and $\sigma_{\beta}$ do not induce a sign change in the blocks $\sigma_{\alpha}^{\otimes i}$ and $\sigma_{\beta}^{\otimes k}$ but do change the sign of the other block respectively. We thus obtain that $R$ changes the sign of $H_{OBC, \alpha,\beta, i,k}$ and is unitary because of the unitarity of the Pauli matrices.
%
%\textbf{Case $i \neq k, i = 1$ or $k = 1$}
%
%Let w.l.o.g. $i > k$. 
%Then we choose $R = \left( \sigma_{\alpha}^{i-1} \otimes \sigma_{\gamma} \right)^{\otimes L \% i}$ where $\gamma \in \{x,y,z\} \setminus \{\alpha, \beta\}$. While both $\sigma_{\alpha}$ and $\sigma_{\gamma}$ change the sign of every term containing $\sigma_{\beta}$, only $\sigma_{\gamma}$ induces a sign change in the terms containing $\sigma_{\alpha}^{\otimes i}$. Thus, we obtain a sign change in the whole Hamiltonian. Again by construction, $R$ is unitary.
%
\textbf{Case $i \neq k$}

Let w.l.o.g. $i > k$.
Then we can construct $R$ as the Kronecker product of $L$ matrices such that at every $k$-th position we apply $\sigma_{\alpha}$ and at every $i$-th position we apply $\sigma_{\beta}$. 
In the case where multiples of $i$ and $k$ coincide, we again choose $\gamma \in \{x,y,z\} \setminus \{\alpha, \beta\}$ and apply it in these positions. 
The remaining free factors are again chosen to be the identity.
%As $R$ is the Kronecker product of unitary matrices, we again obtain that $R$ is unitary. 
It is evident from the construction of $R$ that it induces exactly one sign change in every block $\sigma_{\alpha}^{\otimes i}$ and $\sigma_{\beta}^{\otimes k}$ respectively.

\end{proof}

A different situtation presents itself when we again restrict the interaction terms in the Hamiltonian to involve only one Pauli operator but allow two different interaction lengths.

\begin{double_obc_equal}
	For every spin Hamiltonian with open boundaries of the form
	\[
	H_{OBC, \alpha, i,k} = \sum_{j=0}^{L-i} h_{ij} I^{\otimes j} \otimes \sigma_{\alpha}^{\otimes i} \otimes I^{\otimes L -i -j} + \sum_{l=0}^{L-k} h_{kl} I^{\otimes l} \otimes \sigma_{\alpha}^{\otimes k} \otimes I^{\otimes L -k -l}
	\]
	where $\alpha \in \{x,y,z\}$, $i, k \in 2\mathbb{N}-1, L \in \mathbb{N}$, and $i, k \leq L$, there exists a unitary $R \in \mathbb{C}^{2^L \times 2^L}$ such that $RH_{OBC, \alpha, i, k} = -H_{OBC, \alpha, i, k}R$.
\end{double_obc_equal}

\begin{proof}
	%Like above, we begin by noting that the Pauli matrices anticommute pairwise and that they are unitary. 
We again have to distinguish two cases regarding the relation of $i$ and $k$.

\textbf{Case $i = k$}

In this case, the Hamiltonian is a member of the class considered in Theorem~\ref{sing_obc}.

\textbf{Case $i \neq k$}

%Let w.l.o.g. $i > k$. 
In this case, we can again choose an $\alpha^{\prime }\neq \alpha$ and define $R = \sigma_{\alpha^{\prime}}^{\otimes L}$. 
$R$ then induces an odd number of sign changes in every term of $H_{OBC, \alpha, i,k}$ and consequentially a sign change in the whole Hamiltonian.

\end{proof}

This result can now easily be generalized to the case of more than two interaction lengths for a fixed Pauli operator.

\begin{double_obc_gen}
	By a straight forward generalization of the above proof we obtain that for all Hamiltonians with open boundaries and one Pauli operator of the form 
\begin{align*}
H_{PBC, \alpha} &= \sum_{i \in \mathcal{I}} \sum^{L-i}_{j=0} h_{ij} I^{\otimes j} \otimes \sigma_{\alpha}^{\otimes i} \otimes I^{\otimes L-i-j}
\end{align*}
where $\alpha \in \{x,y,z\}$ and $\mathcal{I} \subset 2\mathbb{N}-1$,  there exists a unitary $R \in \mathbb{C}^{2^L \times 2^L}$ such that $RH_{PBC, \alpha} = -H_{PBC, \alpha}R$.
\end{double_obc_gen}

While we restricted the interaction lengths to be odd for the statements above, we find that there exists another case for arbitrary interaction lengths with a certain relation between them.

\begin{double_obc_div}
	For every spin Hamiltonian with open boundaries of the form
	\[
	H_{OBC, \alpha, i,k} = \sum_{j=0}^{L-i} h_{ij} I^{\otimes j} \otimes \sigma_{\alpha}^{\otimes i} \otimes I^{\otimes L -i -j} + \sum_{l=0}^{L-k} h_{kl} I^{\otimes l} \otimes \sigma_{\alpha}^{\otimes k} \otimes I^{\otimes L -k -l}
	\]
	where $\alpha \in \{x,y,z\}$,$i, k \in \mathbb{N}, i/k \in 2\mathbb{N}-1, L \in \mathbb{N}$, and $i, k \leq L$, there exists a unitary $R \in \mathbb{C}^{2^L \times 2^L}$ such that $RH_{OBC, \alpha, i, k} = -H_{OBC, \alpha, i, k}R$.
	
\end{double_obc_div}

\begin{proof}
	%As for the above statements, we begin by noting that the Pauli matrices anticommute pairwise and that they are unitary. 
Also here, we have to distinguish two cases regarding the relation of $i$ and $k$.
	
\textbf{Case $i = k$}

In this case, the Hamiltonian is a member of the class considered in Theorem~\ref{sing_obc}.

\textbf{Case $i \neq k$}

We can construct $R = \left( I^{\otimes k-1} \otimes \sigma_{\alpha^{\prime}}  \right)^{\otimes L / k} \otimes I^{\otimes L \% k}$ with $\alpha^{\prime }\neq \alpha$ . 
Since $i/k \in 2\mathbb{N}-1$ we find that $R$ induces an odd number of sign changes in all terms of $H_{OBC, \alpha, i,k}$ and thus in the overall Hamiltonian. 
%By construction, it is again unitary.
\end{proof}

What is now left to discuss for Hamiltonians involving up to two different Pauli operators is the case of periodic boundaries, which again introduces more constraints. 
Hence we find that we can only make a positive statement about odd interaction lengths as follows.

\begin{double_pbc_odd}
	For every spin Hamiltonian with periodic boundaries of the form
	\begin{align*}
		H_{PBC, \alpha, \beta, i, l} &= \sum_{j=0}^{L-i} h_{\alpha j} I^{\otimes j} \otimes \sigma_{\alpha}^{\otimes i} \otimes I^{\otimes L -i -j} + \sum_{k=1}^{i-1} h_{\alpha k} \sigma_{\alpha}^{\otimes k} \otimes I^{\otimes L-i} \otimes \sigma_{\alpha}^{i-k} \\
		&+ \sum_{m=0}^{L-l} h_{\beta m} I^{\otimes m} \otimes \sigma_{\beta}^{\otimes l} \otimes I^{\otimes L -l -m} + \sum_{n=1}^{l-1} h_{\beta n} \sigma_{\beta}^{\otimes n} \otimes I^{\otimes L-l} \otimes \sigma_{\beta}^{l-n}
	\end{align*}
	where $\alpha, \beta \in \{x,y,z\}$, $i, l \in 2\mathbb{N}-1$, $L \in \mathbb{N}$ and $i, l \leq L$, there exists a unitary $R \in \mathbb{C}^{2^L \times 2^L}$ such that $RH_{PBC, \alpha, i} = -H_{PBC, \alpha, i}R$.
	
\end{double_pbc_odd}

\begin{proof}
	As before we can choose $\gamma \in \{x,y,z\} \setminus \{\alpha, \beta\}$ and define $R = \sigma_{\gamma}^{\otimes L}$. 
	Since $i, l \in 2\mathbb{N}-1$, $R$ induces and odd number of sign changes in ever term of and consequentially in $H_{PBC, \alpha, \beta, i, l}$.
\end{proof}

This statement can again be readily generalized to multiple interaction lengths and hence more complex Hamiltonians.

\begin{double_pbc_gen}
	By a straight forward generalization of the above proof we obtain that for all Hamiltonians with periodic boundaries and at most two different Pauli operators of the form 
\begin{align*}
H_{PBC, \alpha, \beta} &= \sum_{i \in \mathcal{I}} \sum^{L-i}_{j=0} h_{ij\alpha} I^{\otimes j} \otimes \sigma_{\alpha}^{\otimes i} \otimes I^{\otimes L-i-j} + \sum_{i \in \mathcal{I}} \sum^{i-1}_{k=1} h_{ik\alpha} \sigma_{\alpha}^{\otimes k} \otimes I^{\otimes L-i} \otimes \sigma_{\alpha}^{\otimes i-k}\\  
&+ \sum_{l \in \mathcal{J}} \sum^{L-l}_{j=0} h_{lj\beta} I^{\otimes j} \otimes \sigma_{\beta}^{\otimes l} \otimes I^{\otimes L-l-j} + \sum_{l \in \mathcal{J}} \sum^{l-1}_{k=1} h_{lk\beta} \sigma_{\beta}^{\otimes k} \otimes I^{\otimes L-l} \otimes \sigma_{\beta}^{\otimes l-k}.
\end{align*}
where $\alpha, \beta \in \{x,y,z\}$ and $\mathcal{I}, \mathcal{J} \subset 2\mathbb{N}-1$,  there exists a unitary $R \in \mathbb{C}^{2^L \times 2^L}$ such that $RH_{PBC, \alpha, \beta} = -H_{PBC, \alpha, \beta}R$.
\end{double_pbc_gen}

Now, we finally come to the case of Hamiltonians consisting of interaction terms generated by up to three Pauli operators which is clearly the most complicated setting. 
We begin by inspecting Hamiltonians with open boundaries involving all three Pauli matrices.

\begin{triple_obc_odd}
	For every spin Hamiltonian with open boundaries of the form
	\begin{align*}
		H_{OBC, \alpha, \beta, \gamma, i,k,m} &= \sum_{j=0}^{L-i} h_{\alpha j} I^{\otimes j} \otimes \sigma_{\alpha}^{\otimes i} \otimes I^{\otimes L -i -j} + \sum_{l=0}^{L-k} h_{\beta l} I^{\otimes l} \otimes \sigma_{\beta}^{\otimes k} \otimes I^{\otimes L -k -l} \\
		&+ \sum_{n=0}^{L-m} h_{\gamma n} I^{\otimes n} \otimes \sigma_{\gamma}^{\otimes m} \otimes I^{\otimes L -m -n}
	\end{align*}
	where $\alpha, \beta, \gamma \in \{x,y,z\}, \alpha \neq \beta \neq \gamma \neq \alpha$,$k, L \in \mathbb{N}, i,m \in 2\mathbb{N}-1, i<k$ and $i, k, m \leq L$, there exists a unitary $R \in \mathbb{C}^{2^L \times 2^L}$ such that $RH_{OBC, \alpha, \beta, \gamma, i, k, m} = -H_{OBC, \alpha, \beta, \gamma, i, k, m}R$.
	
\end{triple_obc_odd}

\begin{proof}
	We define $R=\left( \sigma_{\beta}^{k-1} \otimes \sigma_{\alpha} \right)^{\otimes L/k} \otimes \sigma_{\beta}^{\otimes L \% k}$. 
	It is clear that $R$ induces an odd number of sign changes in all blocks $\sigma_{\alpha}^{\otimes i}$ since $i \leq k-1$ is odd. 
	Similarly, it is obvious that $R$ causes exactly one sign change in every block $\sigma_{\beta}^{\otimes k}$ through the single $\sigma_{\alpha}$ in the product. 
	As both $\sigma_{\alpha}$ and $\sigma_{\beta}$ cause sign changes in the blocks $\sigma_{\gamma}^{\otimes m}$ and $m$ is odd, it is evident that $R$ also induces and odd number of sign changes in this case.
	Hence it holds that $RH_{OBC, \alpha, \beta, \gamma, i, k, m} = -H_{OBC, \alpha, \beta, \gamma, i, k, m}R$.

\end{proof}

Finally, we examine the case of two Pauli matrices and three different interaction lengths for open boundary conditions.

\begin{triple_obc_two}
	For every spin Hamiltonian with open boundaries of the form
	\begin{align*}
		H_{OBC, \alpha, \beta,i,k,m} &= \sum_{j=0}^{L-i} h_{\alpha j} I^{\otimes j} \otimes \sigma_{\alpha}^{\otimes i} \otimes I^{\otimes L -i -j} 
		+ \sum_{l=0}^{L-k} h_{\beta kl} I^{\otimes l} \otimes \sigma_{\beta}^{\otimes k} \otimes I^{\otimes L -k -l} \\
		&+ \sum_{n=0}^{L-m} h_{\beta mn} I^{\otimes n} \otimes \sigma_{\beta}^{\otimes m} \otimes I^{\otimes L -m -n}
	\end{align*}
	where $\alpha, \beta \in \{x,y,z\}, \alpha \neq \beta$,$L \in \mathbb{N}, i,k,m \in 2\mathbb{N}-1, i<k$ and $i, k, m \leq L$, there exists a unitary $R \in \mathbb{C}^{2^L \times 2^L}$ such that $RH_{OBC, \alpha, \beta, i, k, m} = -H_{OBC, \alpha, \beta, i, k, m}R$.
\end{triple_obc_two}

\begin{proof}
	As before we can choose $\gamma \in \{x,y,z\} \setminus \{\alpha, \beta\}$ and define $R = \sigma_{\gamma}^{\otimes L}$. 
	Since $i,k, l \in 2\mathbb{N}-1$, $R$ induces and odd number of sign changes in ever term of and consequentially in $H_{PBC, \alpha, \beta, i, l}$.
	
\end{proof}

This statement can now again be generalized to multiple interaction terms.

\begin{triple_obc_gen}
	Again by a straight forward generalization of the above proof we find that for all Hamiltonians with open boundaries and two Pauli operators of the form 
\begin{align*}
H_{PBC, \alpha, \beta} &= \sum_{i \in \mathcal{I}} \sum^{L-i}_{j=0} h_{ij\alpha} I^{\otimes j} \otimes \sigma_{\alpha}^{\otimes i} \otimes I^{\otimes L-i-j}
+\sum_{k \in \mathcal{J}} \sum^{L-k}_{j=0} h_{kj\beta} I^{\otimes j} \otimes \sigma_{\beta}^{\otimes k} \otimes I^{\otimes L-k-j}
\end{align*}
where $\alpha, \beta \in \{x,y,z\}$ and $\mathcal{I} \subset 2\mathbb{N}-1$,  there exists a unitary $R \in \mathbb{C}^{2^L \times 2^L}$ such that $RH_{PBC, \alpha} = -H_{PBC, \alpha}R$.
	
\end{triple_obc_gen}

To the best of our knowledge, we cannot make a positive statement for periodic boundaries and interaction terms involving all three Pauli operators.
As a remark, we would like to point out that in addition to the Hamiltonians treated in this section, positive statements about the existence of an $R$ as considered here should be easy to proof in a very similar way for arbitrary interactions, i.e.\ interactions not between nearest neightbours but arbitrary particles, and odd numbers of particles affected by the interaction terms.
Furthermore, as a special case of the Hamiltonian matrices discussed by Fassbender et.\ al.~\cite{fassbender2006structured} symmetric two-by-two block matrices of the form 
\[
\begin{bmatrix}
	B & C \\
	C & -B
\end{bmatrix}
\]
with $B,C \in \mathbb{R}^{N \times N}$  and $B, C$ symmetric generally exhibit a spectrum symmetric around zero and can thus be considered valid inputs to the presented variant of our method.
This of course is subject to the condition that they yield a sufficiently accurate and small MPO representation.
As a final remark, we note that positive definitive Bethe-Salpeter Hamiltonian matrices in principle also pose a valid input to the algorithm~\cite{benner2017fast, shao2015properties}.

We have shown in this section that a significant subset of all spin Hamiltonians exhibits a point symmetric spectrum around zero according to the introduced characterization and that consequentially there exists a strong use case of our improvement of Algorithm~\ref{approx} in quantum mechanical simulations. 
In the next section, we will now use a well known Hamiltonian belonging to this subset to numerically illustrate the advantage of the improved algorithm in this case.

\section{Numerical Evidence}
\label{numerics}

\begin{figure}[t]
	\centering
	\subfloat {\includegraphics[width=0.49\textwidth]{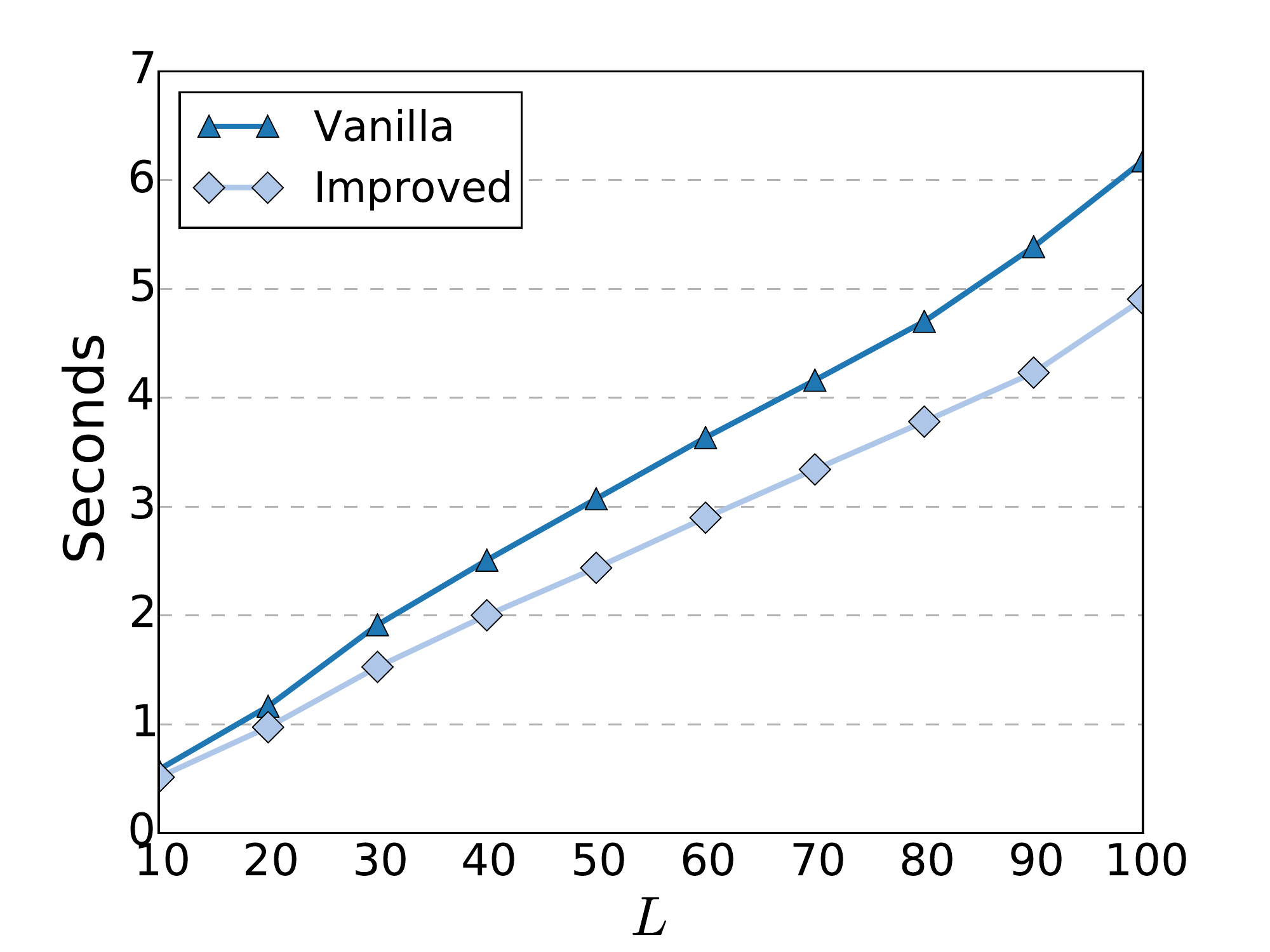}}	
	\subfloat {\includegraphics[width=0.49\textwidth]{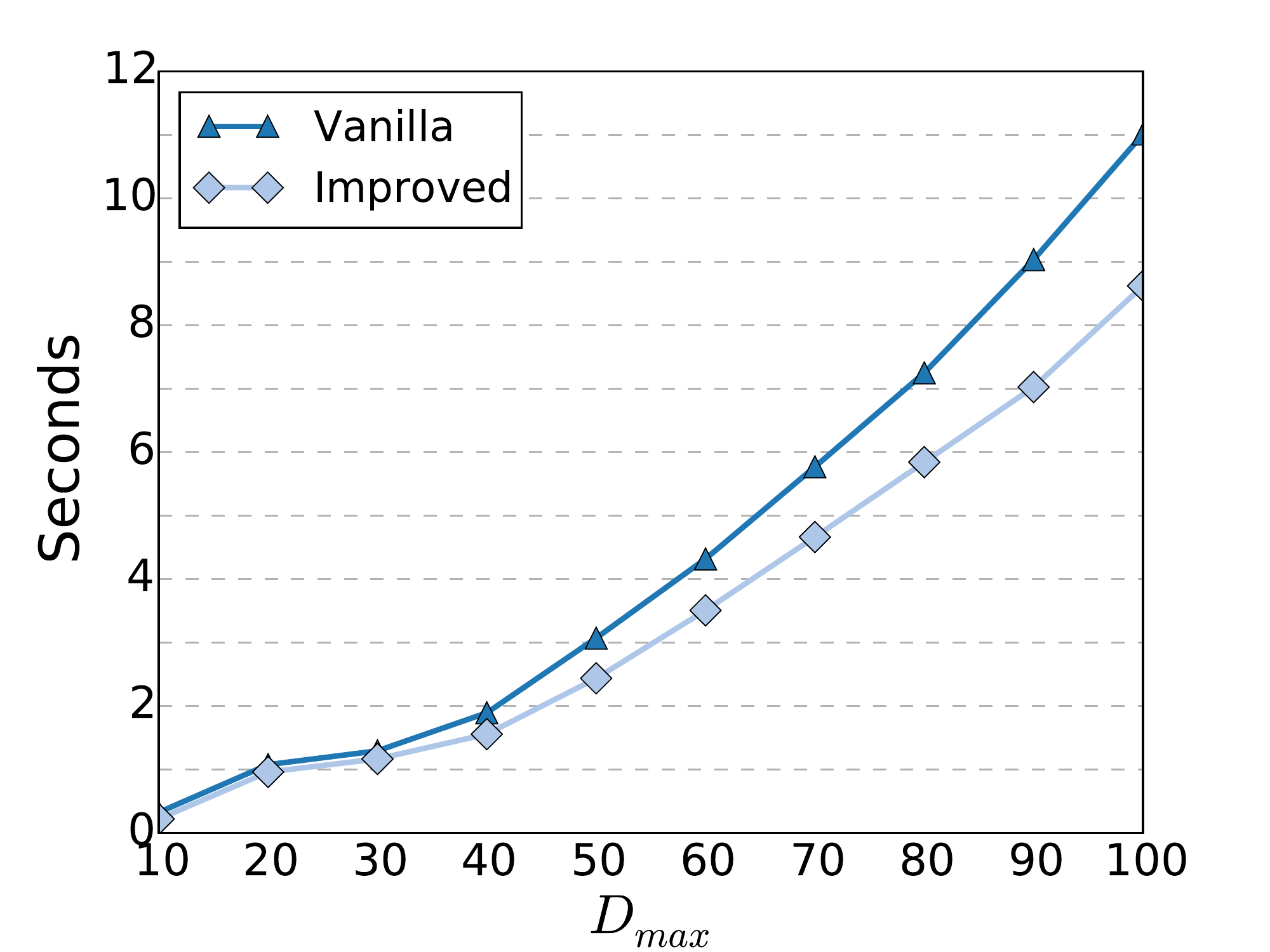}}
	\caption{Comparisons of the runtime in seconds between the improved and the vanilla version of the algorithm. Left: Comparison of average runtime of one iteration over $L$ with $D_{max}=50$. Right: Comparison of average runtime of one iteration over $D_{max}$ with $L=50$.}
\label{fig:runtime}
\end{figure}

\begin{figure}[htpb]
	\centering
	\subfloat {\includegraphics[width=0.49\textwidth]{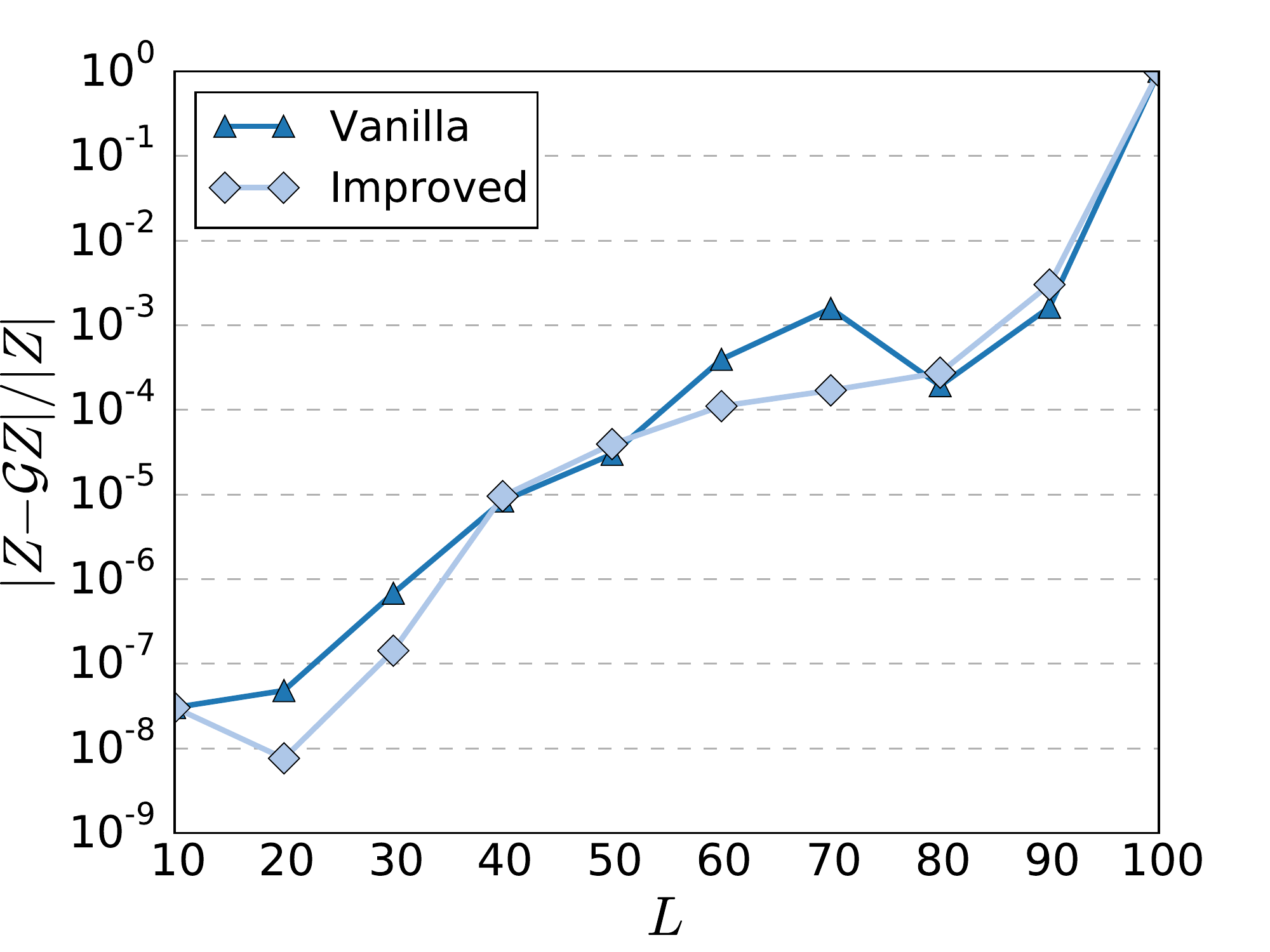}}	
	\subfloat {\includegraphics[width=0.49\textwidth]{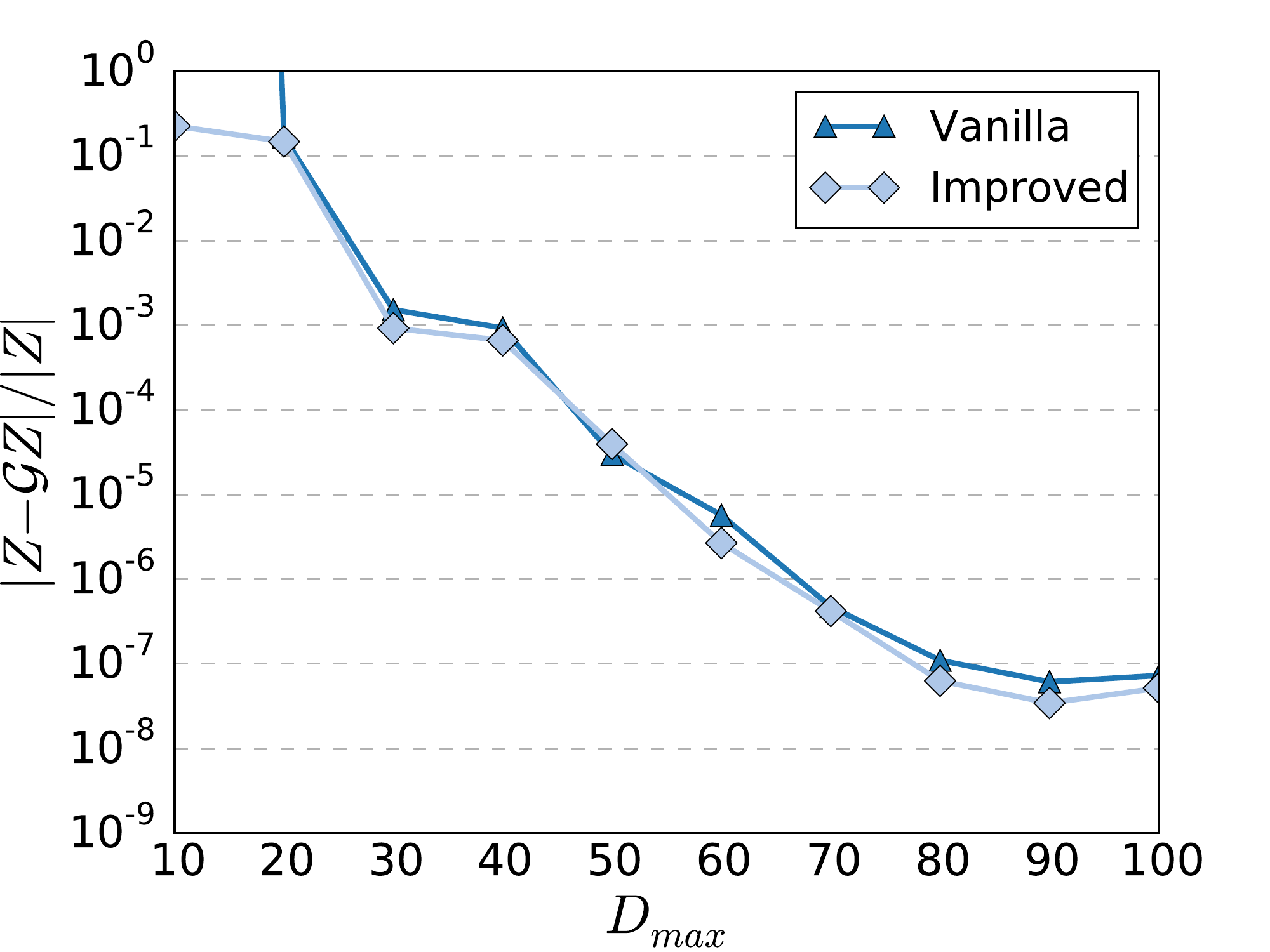}}
	\caption{Comparisons of the relative error in $Z$ between the improved and the vanilla version of the algorithm. Left: Comparison of the relative error over $L$ with $D_{max}=50$. Right: Comparison over the relative error over $D_{max}$ and $L=50$.}
\label{fig:convergence}
\end{figure}

\begin{figure}[htp b]
\center
	\subfloat {\includegraphics[width=0.49\textwidth]{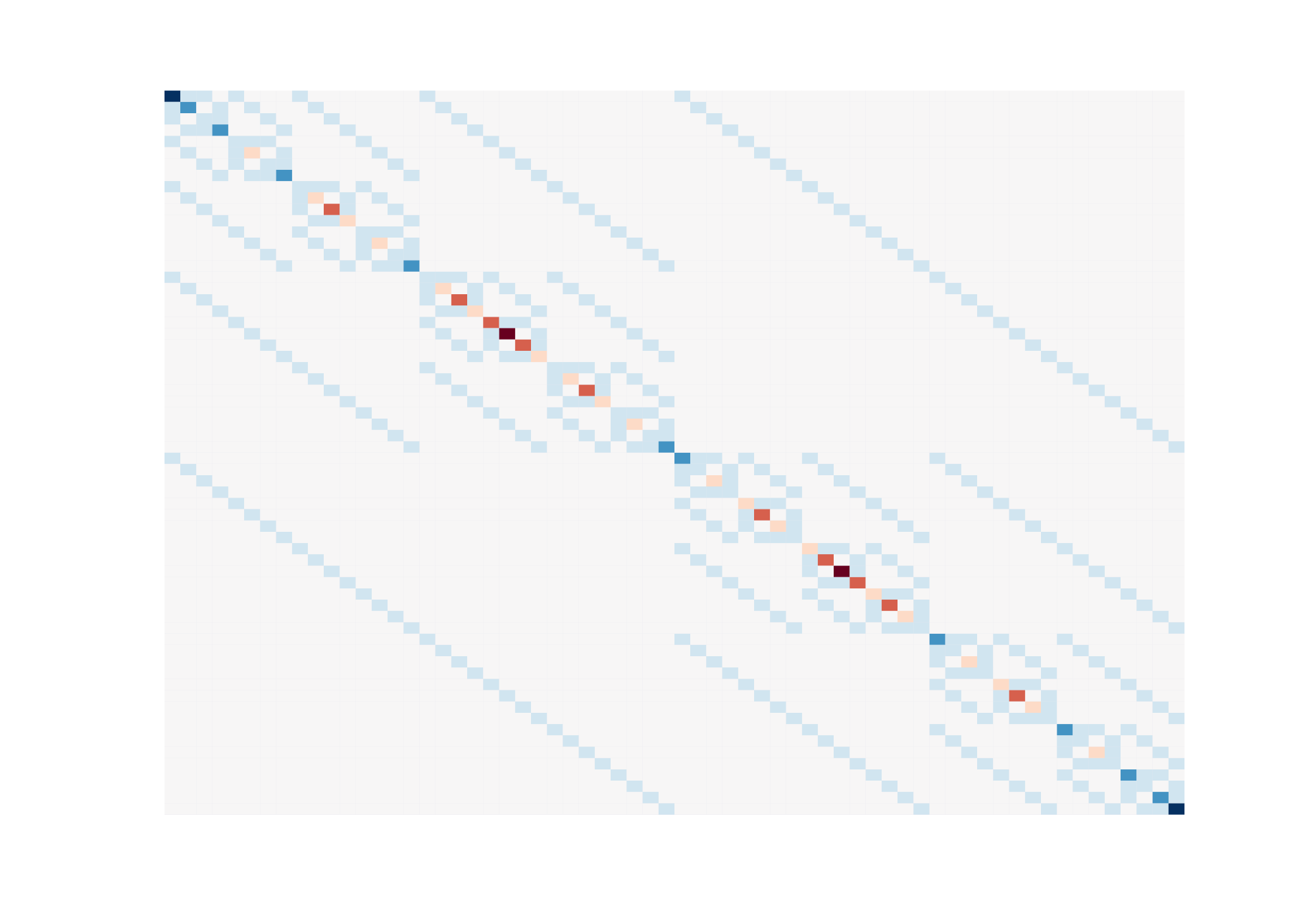}}	
	\subfloat {\includegraphics[width=0.49\textwidth]{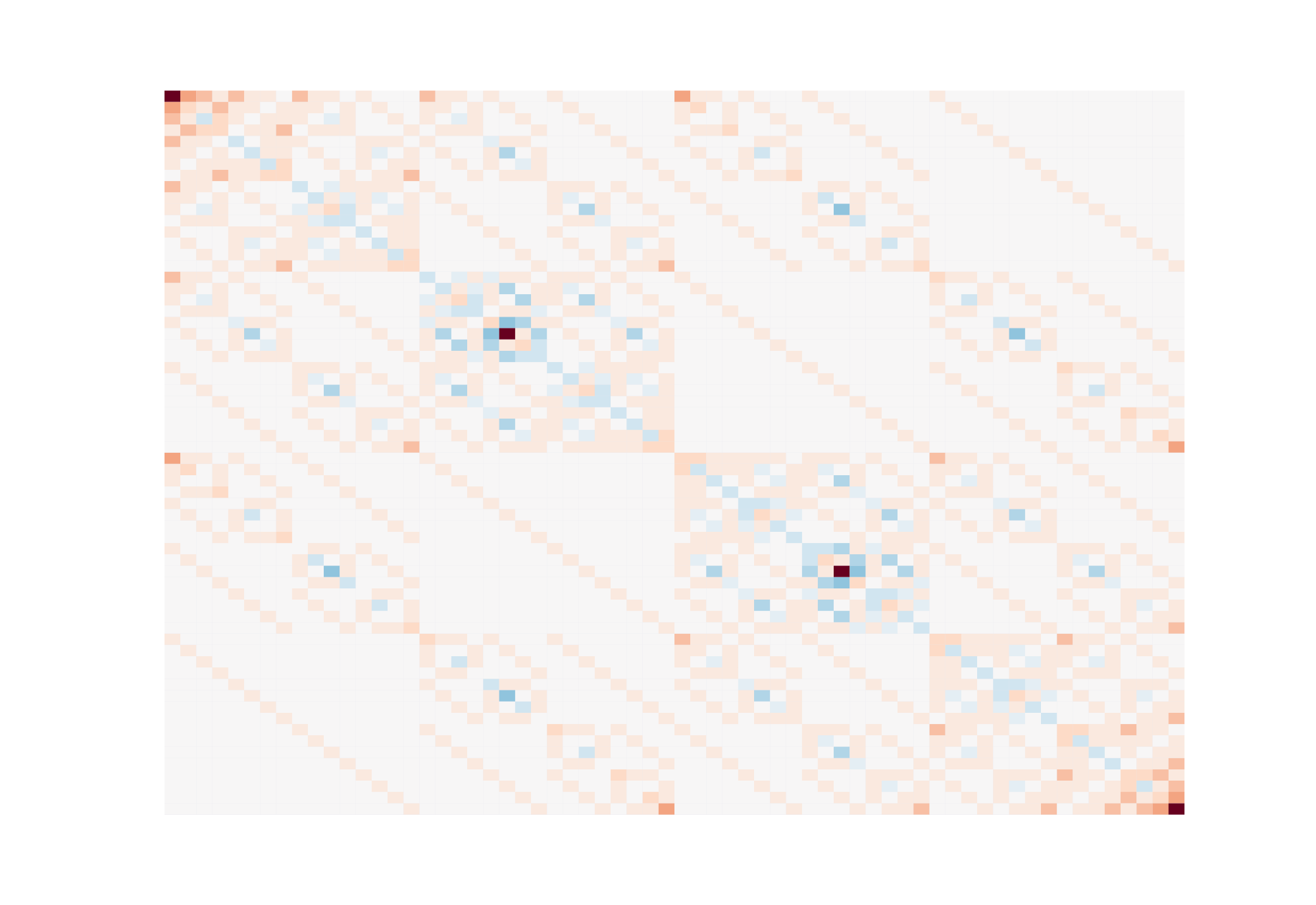}}
	
	\subfloat {\includegraphics[width=0.49\textwidth]{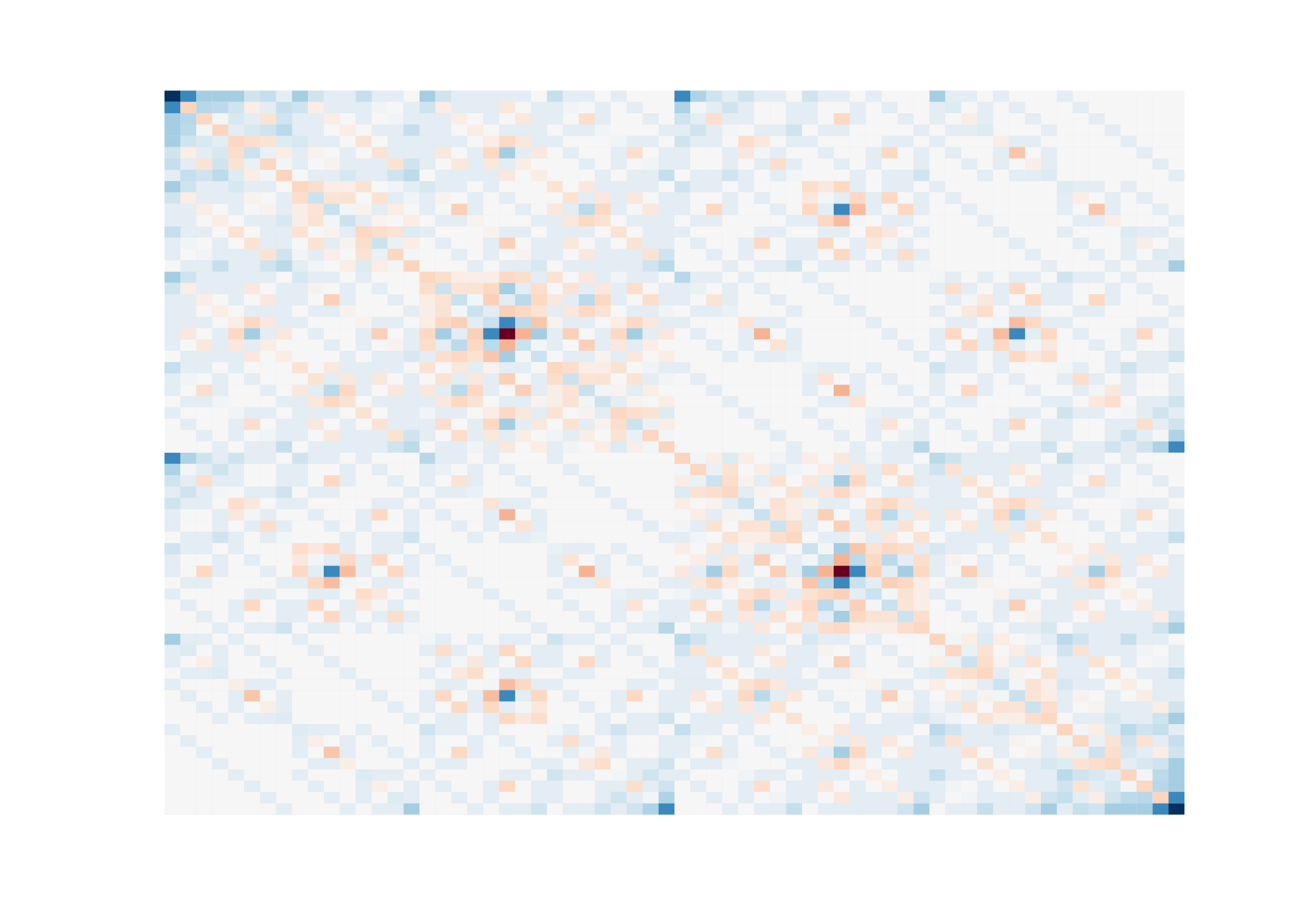}}
	\subfloat {\includegraphics[width=0.49\textwidth]{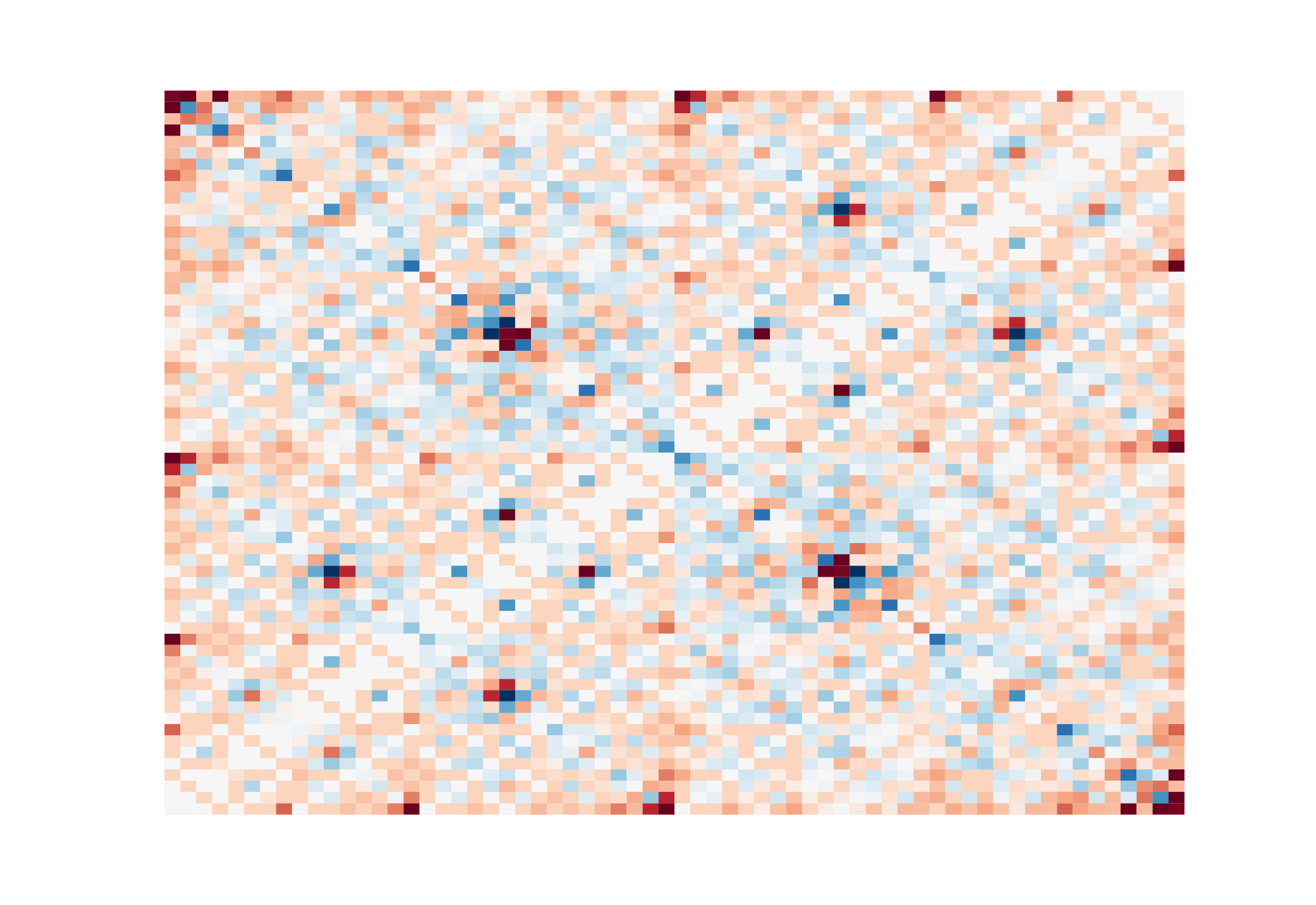}}
	
	\subfloat {\includegraphics[width=0.49\textwidth]{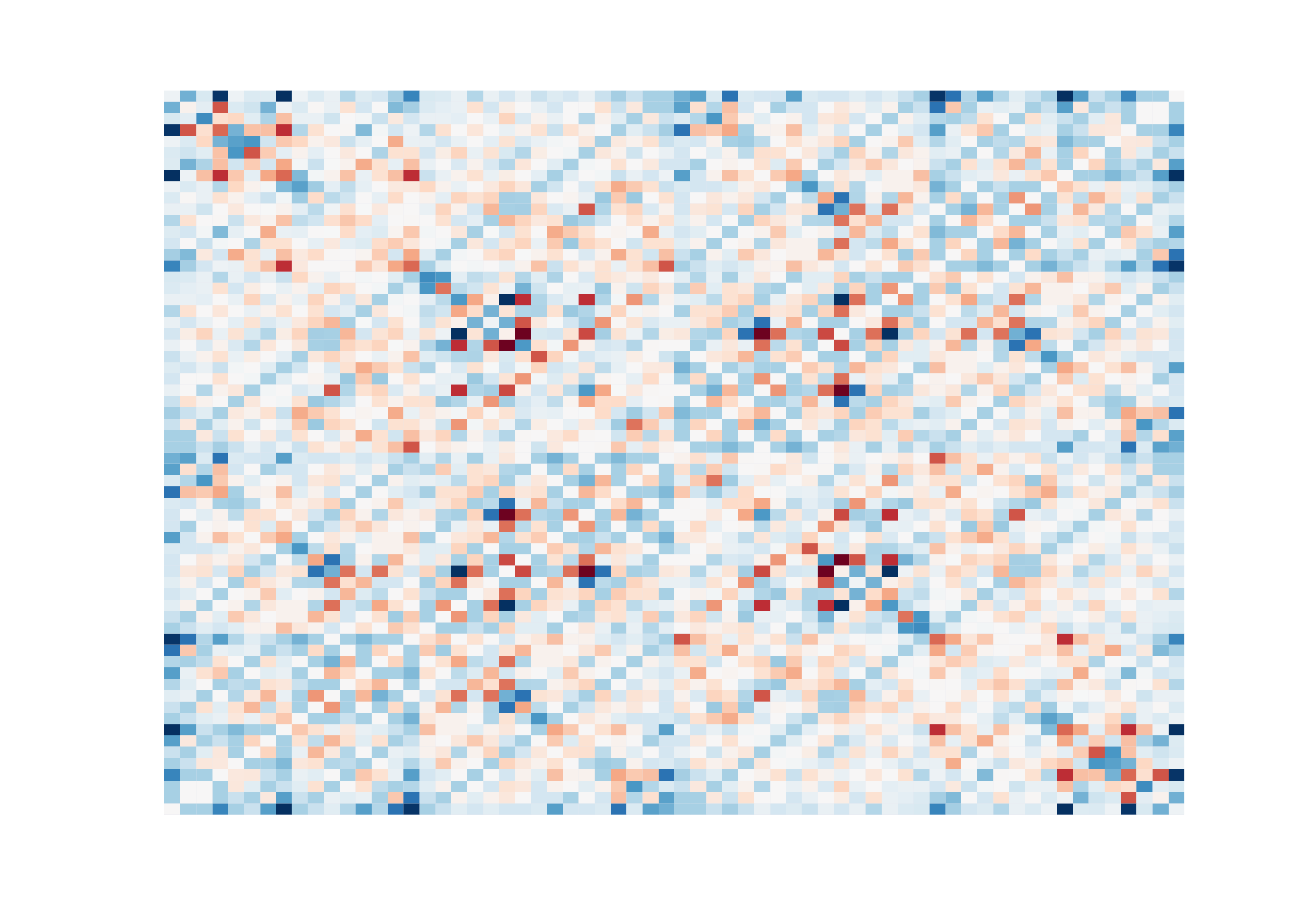}}
	\subfloat {\includegraphics[width=0.49\textwidth]{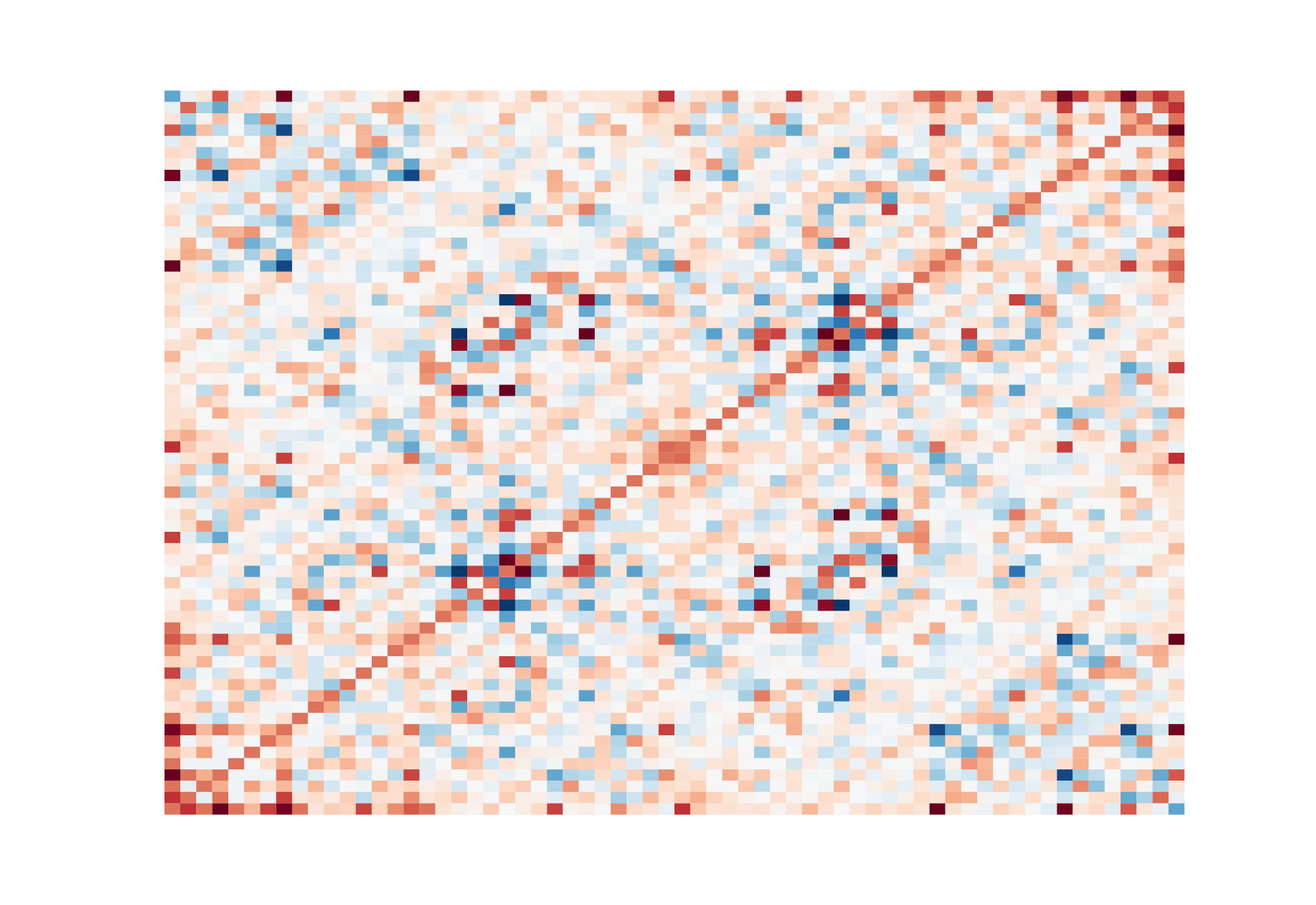}}
	\caption{Heatmaps of the first six basis matrices as computed by the algorithm without approximations for the transverse Ising Hamiltonian with $L=10$.}
\label{symm}
\end{figure}

To provide numerical evidence of the correctness of our statements in Sections~\ref{proofs} and~\ref{hamiltonians}, we will now state results obtained by conducting some numerical experiments for the well known Ising Hamiltonian with a transverse field. 
The Hamiltonian is given by
\[ 
H = J \sum_{i=1}^{L-1} I^{\otimes i-1} \otimes \sigma_x \otimes \sigma_x \otimes I^{\otimes L-(i+1)} + g \sum_{i=1}^{L} I^{\otimes i-1} \otimes \sigma_z \otimes I^{\otimes L - i}
\]
where $\sigma_{x,z}$ are again the Pauli matrices. As we have seen in Section~\ref{hamiltonians}, the transverse field Ising Hamiltonian clearly has the spectral property required to apply the improved version of the algorithm. 
It also has the additional advantage that it can be diagonalized analytically to obtain reference results.
Given a Hamiltonian, its thermal equilibrium, or Gibbs, state is described by
\[
\rho (\beta) = \frac{e^{-\beta H}}{Z}
\]
where $\beta$ is the inverse temperature and 
\[
Z = \trace e^{-\beta H}
\]
is the so called partition function or simply the normalization constant of the distribution. 
As our goal in this section is to compare both versions of the algorithm and not to provide physically relevant results, we will simply approximate $Z$ by chosing 
\[
f(H) = e^{-\beta H}
\]
where we set the scaling coefficients of the Hamiltonian and the inverse temperature to $J=g=\beta=1$.

To provide a thorough comparison between the vanilla, i.e.\ standard, and the improved version of our algorithm in terms of runtime and accuracy, we have conducted two sets of experiments. 
Firstly, we fixed the maximal bond dimension to be $D_{max}=50$ and computed the average runtime for one iteration of the algorithm over a run of 50 iterations for $L$, i.e., the system size, increasing from 10 to 100. 
Secondly, we set $L=50$ and increased $D_{max}$ from 10 to 100 and again computed the average runtime of one iteration over a run of 50 iterations. 
The comparison of the runtimes is depicted in Figure~\ref{fig:runtime}.

For both of these settings, we also evaluated the approximation accuracy as the relative error in $Z$ when we let the algorithm run until the relative difference between approximations results became smaller than $10^{-6}$.
These results are illustrated in Figure~\ref{fig:convergence}.
All results reported here were obtained for a C++ implementation of our algorithm on an Intel i5-5200U mobile CPU.

The results in Figure~\ref{fig:runtime} clearly show an advantage in runtime for the improved version of the algorith for all considered settings. 
On average over all conducted experiments this advantage is around 20\%, which seems like only a modest improvement but can easily amount to several hours of runtime less for large systems and large values of $D_{max}$. 
The results additionally illustrate the linear complexity in $L$ and cubic dependence on $D_{max}$ we have claimed in~\cite{august2017approximation} and which is not affected by the improvement introduced in this work.

In Figure~\ref{fig:convergence} we can furthermore observe that for the case of an input that exhibits the required spectral symmetry the accuracy of both versions of the algorithm is similar with slight advantages for the improved variant. 
This might be due to the fact that the unnecessarily computed partial results in the vanilla version of the algorithm are not exactly zero and hence introduce a small amount of additional error into the approximation.

Finally, in Figure~\ref{symm} we show heatmaps of the first six computed basis matrices in a run of the algorithm without approximations for $L=10$. 
While the first basis matrix is simply the scaled transverse field Ising Hamiltonian, the following matrices represent its orthogonalized powers.
Although this naturally does not constitute a rigorous argument, we can find by simple visual inspection that the basis matrices inherit the symmetric properties of the Hamiltonian, providing some intuition for the statements made in Section~\ref{proofs}.

\section{Conclusion}
\label{conclusion}

In this work we have tried to shed some more light on the analytic properties of the matrix product function approximation algorithm by analyzing the characteristics of the partial results computed during a full run.
As a result, we have found that the basis matrices as computed by the algorithm inherit a range of properties from the input matrix.
We have also seen that these properties then yield a more efficient version of the algorithm for a particular kind of input class, namely the class of matrices with point symmetric spectrum around zero.
We then went on to show for the application of quantum physics that a variety of spin Hamiltonians exhibits this spectral symmetry property and that hence in this field of application the discovered improvement can be successfully applied in many cases.
Finally, we demonstrated and verified our findings in numerical experiments conducted for the example of the Ising Hamiltonian with a transverse magnetic field.

While we were able to improve our understanding of the algorithm, more remains to be done, especially with respect to our understanding of the influence of the introduced truncation errors on the overall approximation accuracy. 
In addition to his, it would be interesting to see further applications of the algorithm outside of numerical quantum physics.
Another possible route of further research would be the exploration of possible combinations of our algorithm with other methods that approximate single extremal eigenvalues.
The approximations of extremal eigenvalues could be used to improve the accuracy of the approximation of the entire spectrum as computed by our algorithm.

\section{Acknowledgements}
We would like to thank the Elite Networks of Bavaria for partly funding this work via the doctoral programme \emph{Exploring Quantum Matter (ExQM)}. 
We would also like to thank Mari Carmen Ba\~{n}uls for being available for helpful discussions.

\nocite{*}
\bibliographystyle{siam}
\bibliography{ref}
\end{document}